\documentclass[twoside]{article}

\usepackage{amsmath,amssymb,amsfonts,amsthm,amscd}
\usepackage{mathtools}
\usepackage{mathrsfs}
\usepackage[mathscr]{eucal}

\usepackage{graphicx}
\usepackage{array,booktabs,multirow}
\usepackage{enumerate}
\usepackage{xspace}
\usepackage{caption}
\usepackage{subfig}
\usepackage{listings}
\usepackage{setspace}
\usepackage{wrapfig}

\usepackage[ruled,vlined]{algorithm2e}

\usepackage{tikz}
\usetikzlibrary{arrows,patterns}

\usepackage{hyperref}
\usepackage{fancyhdr}
\usepackage[all]{xy}
\usepackage{dsfont}

\newcommand{\salt}{\vspace*{2.5mm}}     
\newcommand{\ap}{@}

\usepackage[inner=4cm,outer=4cm]{geometry}

\def\boxit#1{\vbox{\hrule\hbox{\vrule\kern2pt
 \vbox{\kern4pt#1\kern2pt}\kern2pt\vrule}\hrule}}
\def\qed{\hfill
 $\hskip0.3cm
 \boxit{\hsize 2pt \vsize 10pt}$\bigskip\noindent}

\def\C{{\mathchoice {\setbox0=\hbox{$\displaystyle\rm C$}\hbox{\hbox
to0pt{\kern0.4\wd0\vrule height0.9\ht0\hss}\box0}}
{\setbox0=\hbox{$\textstyle\rm C$}\hbox{\hbox
to0pt{\kern0.4\wd0\vrule height0.9\ht0\hss}\box0}}
{\setbox0=\hbox{$\scriptstyle\rm C$}\hbox{\hbox
to0pt{\kern0.4\wd0\vrule height0.9\ht0\hss}\box0}}
{\setbox0=\hbox{$\scriptscriptstyle\rm C$}\hbox{\hbox
to0pt{\kern0.4\wd0\vrule height0.9\ht0\hss}\box0}}}}
\def\Q{{\mathchoice {\setbox0=\hbox{$\displaystyle\rm
Q$}\hbox{\raise
0.15\ht0\hbox to0pt{\kern0.4\wd0\vrule height0.8\ht0\hss}\box0}}
{\setbox0=\hbox{$\textstyle\rm Q$}\hbox{\raise
0.15\ht0\hbox to0pt{\kern0.4\wd0\vrule height0.8\ht0\hss}\box0}}
{\setbox0=\hbox{$\scriptstyle\rm Q$}\hbox{\raise
0.15\ht0\hbox to0pt{\kern0.4\wd0\vrule height0.7\ht0\hss}\box0}}
{\setbox0=\hbox{$\scriptscriptstyle\rm Q$}\hbox{\raise
0.15\ht0\hbox to0pt{\kern0.4\wd0\vrule height0.7\ht0\hss}\box0}}}}
\def\T{{\mathchoice {\setbox0=\hbox{$\displaystyle\rm
T$}\hbox{\hbox to0pt{\kern0.3\wd0\vrule height0.9\ht0\hss}\box0}}
{\setbox0=\hbox{$\textstyle\rm T$}\hbox{\hbox
to0pt{\kern0.3\wd0\vrule height0.9\ht0\hss}\box0}}
{\setbox0=\hbox{$\scriptstyle\rm T$}\hbox{\hbox
to0pt{\kern0.3\wd0\vrule height0.9\ht0\hss}\box0}}
{\setbox0=\hbox{$\scriptscriptstyle\rm T$}\hbox{\hbox
to0pt{\kern0.3\wd0\vrule height0.9\ht0\hss}\box0}}}}
\def\Z{{\mathchoice {\hbox{$\sf\textstyle Z\kern-0.4em Z$}}
{\hbox{$\sf\textstyle Z\kern-0.4em Z$}}
{\hbox{$\sf\scriptstyle Z\kern-0.3em Z$}}
{\hbox{$\sf\scriptscriptstyle Z\kern-0.2em Z$}}}}

\newcommand{\eqnoone}  
   {}
\newcommand{\eqnotwo}  
   {}

\newcounter{alf}

\newcommand{\adresa}[1]{\par\vspace*{-11pt}
                        \begin{flushright}
                        {\small
                        #1}
                        \end{flushright}
                        }

\newcommand{\cC}{\mathcal{C}}

\newcommand{\cP}{\mathcal{P}}

\newcommand{\cG}{\mathcal{G}}

\newtheorem{theorem}{Theorem}[section]

\newtheorem{proposition}[theorem]{Proposition}

\begin{document}


  \vskip 1.2 true cm
\setcounter{page}{1}

\begin{center} 
{\bf The Complete Intersection property for binomial ideals of collections of cells} \\

          {by}\\
          
{\sc Rodica Dinu$^{(1)}$, Francesco Navarra$^{(2)}$}

\end{center}

\markboth{Complete Intersection binomial ideals of collections of cells}{R. Dinu, F. Navarra}

\begin{center}
{\em In the memory of Ionel Bucur (1930-1976) and Nicolae Radu (1931-2001)}
\end{center}
\vspace{.1in}

\begin{abstract}
    In this paper, we provide a combinatorial characterization of those collections of cells whose inner $2$-minor ideals are complete intersections. More precisely, given a collection of cells $\cC$ and its associated inner $2$-minor ideal $I_\cC$, we prove that $I_\cC$ is a complete intersection if and only if $\cC$ is a chessboard.
\end{abstract}

\begin{quotation}

\end{quotation}

\thispagestyle{empty}

\section{Introduction}\label{Sec1}

Let $X$ be an $m \times n$ matrix of indeterminates over a field $K$.  
For each integer $1 \le t \le \min\{m,n\}$, the ideal generated by all $t$--minors of $X$, known as a \emph{determinantal ideal}, has been extensively studied and constitutes a classical subject in Commutative Algebra and Algebraic Geometry.  
A comprehensive reference on determinantal ideals is the monograph by Bruns and Vetter~\cite{BV}.

Over the years, several generalizations of determinantal ideals have been proposed in which only specific subsets of $t$--minors are considered. 
A prominent example arises when the selected minors follow a prescribed combinatorial pattern, such as a ladder shape, leading to the notions \textit{ladder} determinantal ideals; see \cite{conca3}.  

In contrast, when $I$ is generated by an arbitrary collection of $t$--minors of $X$, the structure of the ideal becomes significantly more intricate, even in the first nontrivial case $t=2$. 
Considerable effort has therefore been devoted to investigating the algebraic properties of ideals generated by arbitrary sets of $2$--minors of $X$. 
This line of research is often framed in combinatorial terms, interpreting such ideals via graphs or collections of cells and analyzing the interplay between their algebraic and combinatorial properties. 
Notable examples include binomial edge ideals, ideals generated by adjacent $2$--minors, and inner $2$--minor ideals associated with collections of cells. A comprehensive treatment of these binomial ideals can be found in \cite{HHO_Book}.

The systematic study of inner $2$--minor ideals of collections of cells from the viewpoint of Commutative Algebra was initiated in \cite{Qureshi2012}. 
Given a collection of cells $\cC$, one associates to it the binomial ideal $I_{\cC}$ generated by its inner $2$--minors in the polynomial ring 
$S_{\cC} = K[x_a : a \in V(\cC)]$, where $V(\cC)$ denotes the vertex set of $\cC$ and $K$ is a field. 
The quotient ring $K[\cC] = S_{\cC}/I_{\cC}$ is called the coordinate ring of $\cC$.

A central problem is to determine when such an ideal is \emph{prime} or 
\emph{radical}, and to describe its primary decomposition and height in terms of 
the combinatorial structure of the underlying collection of cells. 
Significant progress in this direction has been achieved; we refer to 
\cite{CistoNavarra2021, CNU1, KoleyKotalVeer2024, MRR2020, QSS, ZGW1} for representative results.
 
More recently, a new line of research has emerged with the aim of describing 
invariants such as the Hilbert--Poincar\'e series and the 
Castelnuovo--Mumford regularity of the coordinate ring in terms of 
rook-theory invariants. Look at \cite{DinuNavarra2023, EneHerzogQureshiRomeo2021, 
JahangirNavarra2024, KV, Navarra2025, QureshiRinaldoRomeo2022, RR}, for some of the most significant contributions.

A complete bibliography on the aforementioned property, as well as on further algebraic and homological properties of these ideals, comprising around forty research articles, can be found in the recent survey \cite{Survey}.

In this paper, we investigate a property of inner $2$--minor ideals that has not 
been addressed so far, that is, the \textit{complete intersection} property.

Recall that a homogeneous ideal $I$ in $K[x_1,\dots,x_n]$ is called 
a \emph{complete intersection} if it is generated by a regular sequence of 
homogeneous polynomials. Equivalently, this means that $I$ is a 
complete intersection if and only if
\[
\mu(I) = \operatorname{ht}(I),
\]
where 
\begin{itemize}
    \item $\mu(I)$ denotes the cardinality of a minimal homogeneous generating set of $I$, and
    \item $\operatorname{ht}(I)$ denotes the height of $I$.
\end{itemize}
This numerical characterization will be adopted throughout the paper.
We refer to \cite{BH} for further background on complete intersections.

The main objective of this work is to provide a combinatorial characterization 
of those collections of cells whose inner $2$--minor ideals are complete 
intersections.

The paper is organized as follows. In Section~\ref{Section: Preliminaries}, we review the necessary background on collections of cells and their associated binomial ideals. Section~\ref{Sec1} is devoted to the proof of our main characterization theorem, namely Theorem~\ref{Thm}, which states that the inner $2$--minor ideal of a collection of cells $\cC$ is a complete intersection if and only if $\cC$ is a chessboard. Roughly speaking, a chessboard consists of cells attached vertex-to-vertex (see Figure~\ref{Figure: weakly conn. coll. of cells + chessboard}(b)). We first prove a sufficient condition in Proposition~\ref{Prop: chessboard implies CI} by constructing a suitable monomial order under which the set of generators of the inner $2$--minor ideal of a chessboard forms a reduced Gr\"obner basis. We then establish a necessary condition in Proposition~\ref{Prop: CI implies chessboard}, using the upper bound $\operatorname{ht}(I_{\cC}) \le |\cC|$ provided in Proposition~\ref{Prop: LC is a minial prime}, valid for an arbitrary collection of cells $\cC$.


\section{Collections of cells and Commutative Algebra}\label{Section: Preliminaries}

In this section, we introduce the terminology and notation concerning collections of cells and their associated ideals of $2$-minors.

Let $(i,j), (k,l) \in \mathbb{Z}^2$. We define a partial order on $\mathbb{Z}^2$ by setting $(i,j) \leq (k,l)$ if and only if $i \leq k$ and $j \leq l$. For $a = (i,j)$ and $b = (k,l)$ in $\mathbb{Z}^2$ with $a \leq b$, we define
\[
[a,b] = \{(m,n) \in \mathbb{Z}^2 \mid i \leq m \leq k,\ j \leq n \leq l\}
\]
and call it an \emph{interval} in $\mathbb{Z}^2$. If $i < k$ and $j < l$, then $[a,b]$ is said to be a \emph{proper} interval. In this case, $a$ and $b$ are called the \emph{diagonal corners} of $[a,b]$, and the points $c = (i,l)$ and $d = (k,j)$ are called the \emph{anti-diagonal corners}. If $j = l$ (respectively, $i = k$), then $a$ and $b$ are said to be in \emph{horizontal} (respectively, \emph{vertical}) position.

A proper interval $C = [a,b]$ with $b = a + (1,1)$ is called a \emph{cell} in $\mathbb{Z}^2$. The points $a$, $b$, $c$, and $d$ are referred to as the \emph{lower-left}, \emph{upper-right}, \emph{upper-left}, and \emph{lower-right} corners of $C$, respectively. We denote the sets of \emph{vertices} and \emph{edges} of $C$ by
\[
V(C) = \{a, b, c, d\}, 
\qquad 
E(C) = \{\{a,c\}, \{c,b\}, \{b,d\}, \{a,d\}\}.
\]
For a collection of cells $\cP$ in $\mathbb{Z}^2$, we set
\[
V(\cP) = \bigcup_{C \in \cP} V(C), 
\qquad 
E(\cP) = \bigcup_{C \in \cP} E(C).
\]
The \emph{rank} of $\cP$, denoted by $|\cP|$, is the number of cells in $\cP$. In literature is also used the notation: $\mathrm{rank}(\cC)$.

Let $A$ and $B$ be two cells in $\mathbb{Z}^2$ with lower-left corners $a = (i,j)$ and $b = (k,l)$, respectively, and assume that $a \leq b$. The \emph{cell interval} $[A,B]$, also called a \emph{rectangle}, is the set of all cells in $\mathbb{Z}^2$ whose lower-left corners $(r,s)$ satisfy $i \leq r \leq k$ and $j \leq s \leq l$.

Let $\cP$ be a collection of cells. An interval $[a,b]$ of $\mathbb{Z}^2$ is said to be the \emph{minimal bounding rectangle} of $\cP$ if $V(\cP) \subseteq [a,b]$ and no proper interval $[a',b']$ of $\mathbb{Z}^2$ such that $V(\cP)\subsetneq [a',b'] \subsetneq [a,b]$. 


A finite collection of cells $\cP$ is said to be \emph{weakly connected} if, for any two cells $C$ and $D$ in $\cP$, there exists a sequence of cells $\mathcal{C} \colon C = C_1, \dots, C_m = D$ in $\cP$ such that $V(C_i) \cap V(C_{i+1}) \neq \emptyset$ for all $i = 1, \dots, m-1$. See Figure \ref{Figure: weakly conn. coll. of cells + chessboard} (a) for an example.

A collection of cells is called a \emph{chessboard} if each connected component consists of cells and any two distinct cells intersect in at most one vertex. See Figure \ref{Figure: weakly conn. coll. of cells + chessboard} (b) for an example.

 \begin{figure}[h!]
     \centering
     \subfloat[]{\includegraphics[scale=0.2]{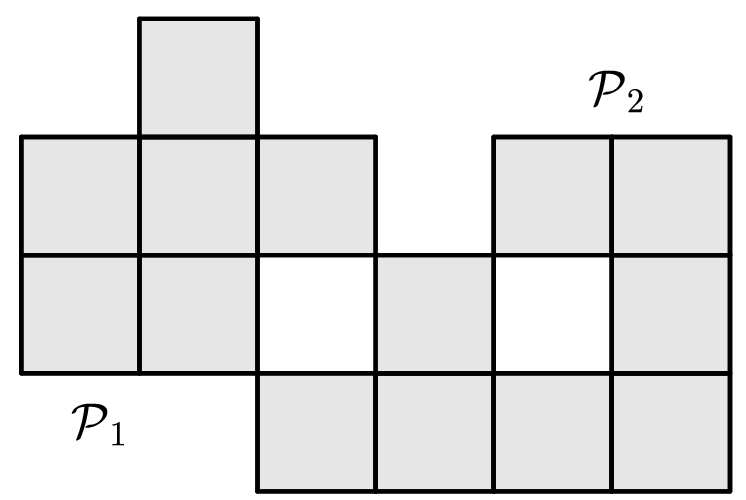}}\qquad\qquad
     \subfloat[]{\includegraphics[scale=0.2]{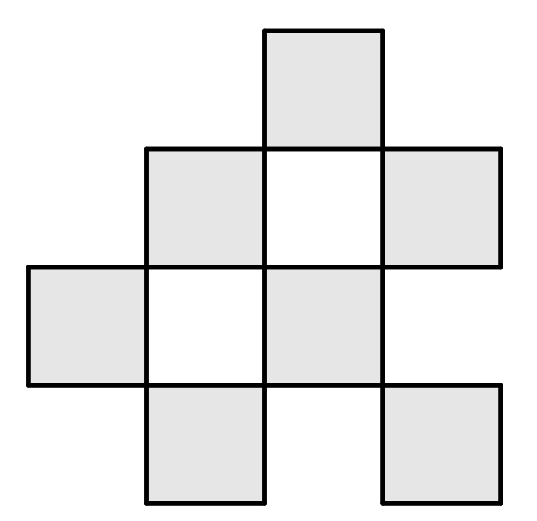}}
     \caption{A weakly connected collection of cells and a chessboard.}
\label{Figure: weakly conn. coll. of cells + chessboard}
\end{figure}

Let $\cP$ be a finite collection of cells, and let $S_\cP = K[x_v : v \in V(\cP)]$ be the polynomial ring associated with $\cP$, where $K$ is a field. A proper interval $[a,b]$ is called an \emph{inner interval} of $\cP$ if all cells in $\cP_{[a,b]}$ belong to $\cP$. If $[a,b]$ is an inner interval of $\cP$, with diagonal corners $a$ and $b$ and anti-diagonal corners $c$ and $d$, then the binomial $x_a x_b - x_c x_d$ is called an \emph{inner $2$-minor} of $\cP$. The ideal $I_{\cP} \subset S_\cP$ generated by all inner $2$-minors of $\cP$ is called the \emph{ideal of $2$-minors} of $\cP$. The quotient ring $K[\cP] = S_\cP / I_{\cP}$ is called the \emph{coordinate ring} of $\cP$. The minimal set of generators of $I_\cC$ is denoted by $\cG(I_{\cC})$.

\section{A combinatorial characterization of the Complete Intersection property}

This section is devoted to the main result of the paper, which is the following characterization.

\begin{theorem}\label{Thm}
Let $\cC$ be a collection of cells. Then $I_\cC$ is a complete intersection ideal if and only if $\cC$ is a chessboard.
\end{theorem}

We begin by proving a sufficient condition for the complete intersection property, which is stated in Proposition~\ref{Prop: chessboard implies CI}. The necessary condition will be established subsequently in Proposition~\ref{Prop: CI implies chessboard}.

\begin{proposition}\label{Prop: chessboard implies CI}
Let $\cC$ be a collection of cells. If $\cC$ is a chessboard, then $I_{\cC}$ is a complete intersection ideal.
\end{proposition}

\begin{proof}
We may assume that $\cC$ is weakly connected. Indeed, suppose that $\cC$ decomposes into weakly connected collections of cells $\cC_1, \dots, \cC_r$, and that $V(\cC)$ is the disjoint union of the sets $V(\cC_j)$ for $j = 1, \dots, r$. Then
\[
I_{\cC} = I_{\cC_1} + \cdots + I_{\cC_r}
\]
in a polynomial ring whose variables decompose according to the disjoint union of the vertex sets. In this situation, $I_{\cC}$ is a complete intersection if and only if each $I_{\cC_j}$ is a complete intersection. Hence, it suffices to treat the case in which $\cC$ is weakly connected.

Assume therefore that $\cC$ is a weakly connected chessboard. We recall that if the initial ideal of an homogeneous ideal $I$ with respect to some monomial order is a complete intersection, then $I$ itself is a complete intersection. Hence, we aim to define a suitable monomial order $<$ on $S_\cC$ such that $\operatorname{in}_<(I_\cC)$ is a complete intersection, i.e., generated by a regular sequence; equivalently, for every pair $m_1, m_2$ in the minimal generating set of $\operatorname{in}_<(I_\cC)$, we have $\operatorname{gcd}(m_1,m_2)=1$.

Let $[a,b]$ be the minimal bounding rectangle of $\cC$. Without loss of generality, after a suitable translation, we may assume that $a = (0,0)$ and $b = (m,n)$. We define a total order $<_{V(\cC)}$ on $V(\cC)$ as follows.

For $(i,j)$ and $(k,l)$ in $V(\cC)$, we set $(i,j) <_{V(\cC)} (k,l)$ if one of the following holds:
\begin{enumerate}
    \item $j < l$; or
    \item $j = l$, $j$ is even, and $i < k$; or
    \item $j = l$, $j$ is odd, and $i > k$.
\end{enumerate}

We then let $<$ denote the lexicographic order on $S_\cC$ induced by the variable order $x_p < x_q$ whenever $p <_{V(\cC)} q$.

For instance, if we consider the chessboard $\cC'$ in Figure~\ref{Figure: weakly conn. coll. of cells + chessboard} (b), then $<$ is the lexicographic order determined by
\[
x_{20} < \dots < x_1,
\]
where the vertices $\{1,\dots,20\}$ are depicted in Figure~\ref{Figure: order chessboard} (a). In Figure~\ref{Figure: order chessboard} (b), the diagonal or anti-diagonal segments represent the leading terms of the binomial generators of $I_{\cC'}$.
 \begin{figure}[h!]
     \centering
     \subfloat[]{\includegraphics[scale=0.2]{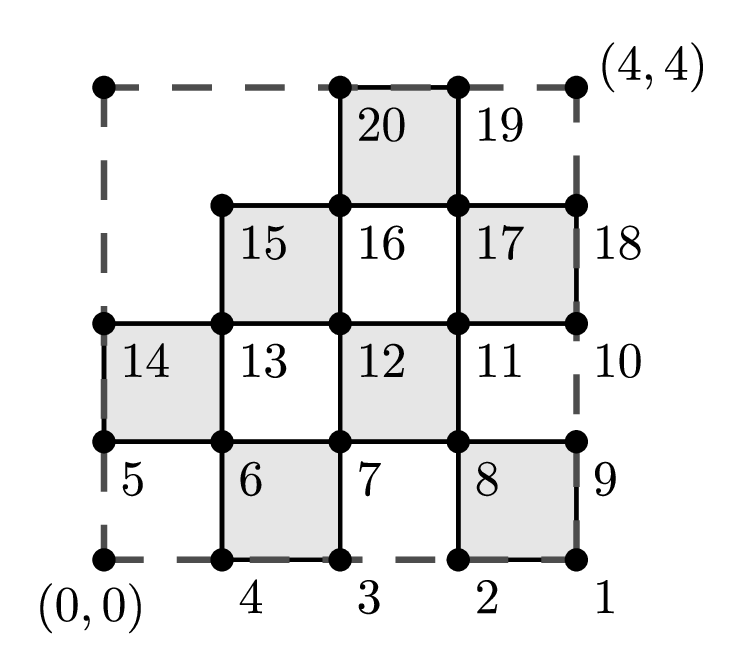}}\qquad
     \subfloat[]{\includegraphics[scale=0.2]{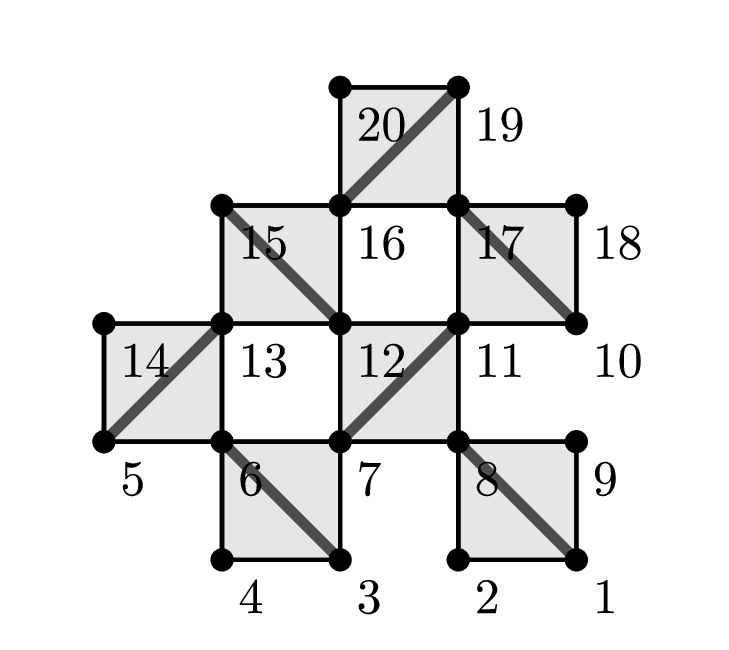}}
     \caption{Order of vertices in a chessboard.}
\label{Figure: order chessboard}
\end{figure}

In general, consider two cells in $\cC$ sharing a common vertex. There are two possible configurations.

\begin{enumerate}
    \item The cells are arranged as in Figure~\ref{Figure: order chessboard proof1}. 

       \begin{figure}[h]
     \centering
     \subfloat[]{\includegraphics[scale=0.2]{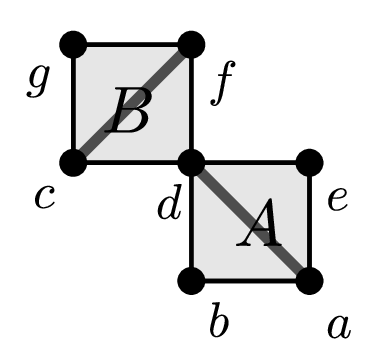}}\qquad\qquad
     \subfloat[]{\includegraphics[scale=0.2]{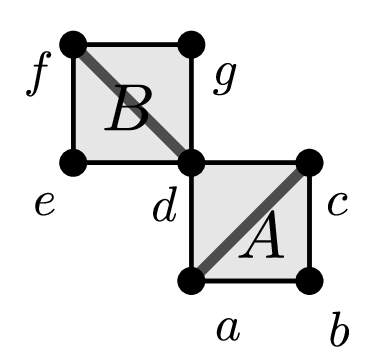}}
     \caption{Order in a part of a chessboard.}
\label{Figure: order chessboard proof1}
\end{figure}
    
    The order $<_{V(\cC)}$ is given by $g<f<e<c<b<a$, as shown in the figure. Let $h_A$ and $h_B$ be the binomials in $\cG(I_{\cC})$ corresponding to $A$ and $B$, respectively. The initial terms of $h_A$ and $h_B$ do not share any variables; hence,
    \[
    \operatorname{gcd}(\operatorname{in}_<(h_A), \operatorname{in}_<(h_B)) = 1.
    \]
 
    \item The cells are arranged as in Figure~\ref{Figure: order chessboard proof2}. This configuration is obtained from the previous one by reflection along the $y$-axis. The same argument as in case (1) applies.
    
     \begin{figure}[h]
     \centering
     \subfloat[]{\includegraphics[scale=0.2]{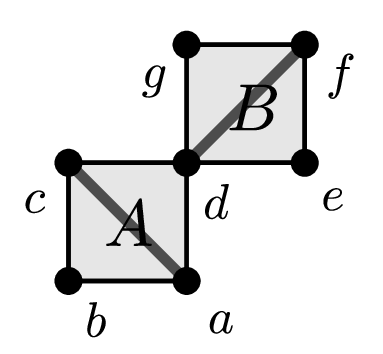}}\qquad\qquad
     \subfloat[]{\includegraphics[scale=0.2]{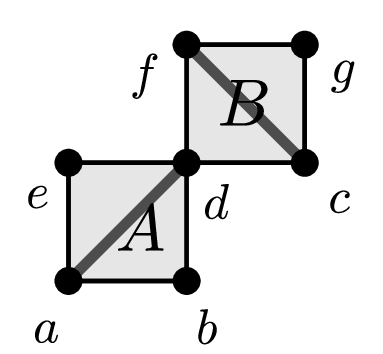}}
     \caption{Order in a part of a chessboard.}
\label{Figure: order chessboard proof2}
\end{figure}
\end{enumerate}

Therefore, for every pair $h_1, h_2 \in \cG(I_\cC)$, we have
\[
\operatorname{gcd}(\operatorname{in}_<(h_1), \operatorname{in}_<(h_2)) = 1.
\]
It follows that $\operatorname{in}_<(I_\cC)$ is a complete intersection. Consequently, $I_\cC$ is a complete intersection.
\end{proof}

Having established the sufficient condition, we now turn to the necessary condition for a collection of cells to yield a complete intersection ideal. \\

We begin by comparing the height of an inner $2$-minor ideal with the rank of the related collection of cells, in particular by proving that $\operatorname{ht}(I_\cC) \leq |\cC|$. This inequality was very recently established in \cite{HHMoradi} via a characterization of the height of a binomial ideal $I \subset S$ in terms of the dimension of the $\mathbb{Q}$-vector space spanned by the set of integer vectors $\{\textbf{v}-\textbf{w} \in \mathbb{Q}^n : \textbf{x}^\textbf{v} - \textbf{x}^\textbf{w} \in I\}$. Here, we present a short and direct proof by showing that an appropriate lattice ideal associated with $\cC$ is a minimal prime of $I_\cC$.

We begin by recalling the definition of the lattice ideal of a collection of cells, as given in \cite{Qureshi2012}. 

Let $\cC$ be a collection of cells. For each $a \in V(\cC)$, denote by $\mathbf{v}_a$ the vector in $\mathbb{Z}^{|V(\cC)|}$ having $1$ in the coordinate indexed by $a$ and $0$ in all other coordinates. If $C=[a,b] \in \cC$ is a cell with $a,b$ as its diagonal corners and $c,d$ as its anti-diagonal corners, we set $\mathbf{v}_{[a,b]} = \mathbf{v}_a + \mathbf{v}_b - \mathbf{v}_c - \mathbf{v}_d \in \mathbb{Z}^{|V(\cC)|}$. We define $\Lambda_\cC$ as the sublattice of $\mathbb{Z}^{|V(\cC)|}$ generated by the vectors $\mathbf{v}_C$ for all $C \in \cC$.  Observe that the rank of $\Lambda_\cC$ is equal to $\vert \cC\vert$.\\
Let $n = |V(\cC)|$. For $\mathbf{v} \in \mathbb{N}^n$, we denote by $x^{\mathbf{v}}$ the monomial in $S_\cC$ having $\mathbf{v}$ as its exponent vector. For $\mathbf{e} \in \mathbb{Z}^n$, we denote by $\mathbf{e}^+$ the vector obtained from $\mathbf{e}$ by replacing all negative components with zero, and by $\mathbf{e}^- = -(\mathbf{e} - \mathbf{e}^+)$ the corresponding non-positive part.  \\
Let $L_\cC$ be the lattice ideal of $\Lambda_\cC$, that is, the following binomial ideal in $S_\cC$:
\[
L_\cC = (\{x^{\mathbf{e}^+} - x^{\mathbf{e}^-} \mid \mathbf{e} \in \Lambda_\cC\}).
\]

Given a collection of cells $\cC$, the ideal $L_\cC$ is a prime ideal, see \cite[pp. 288]{Qureshi2012}. Moreover, it is shown in \cite[Theorem~3.6]{Qureshi2012} that $I_\cC$ is a prime ideal if and only if $I_\cC = L_\cC$. 

\begin{proposition}\label{Prop: LC is a minial prime}
Let $\cC$ be a collection of cells. Then $L_{\cC}$ is a minimal prime ideal of $I_{\cC}$. Moreover, $\mathrm{ht}(I_{\cC}) \leq |\cC|$.
\end{proposition}

\begin{proof}
Assume that there exists a prime ideal $\mathfrak{p}$ satisfying 
$I_{\cC} \subseteq \mathfrak{p} \subseteq L_{\cC}$. 
We prove that $L_{\cC} \subseteq \mathfrak{p}$.

Let $f \in L_{\cC}$. Denote by $I_{\mathrm{adj}}(\cC)$ the ideal of $S_{\cC}$ generated by all binomials of the form $x_a x_b - x_c x_d$, where $[a,b]$ is a cell of $\cC$ and $c,d$ are its anti-diagonal corners. By definition, one has 
$I_{\mathrm{adj}}(\cC) \subseteq I_{\cC}$. 
According to \cite[Proposition~1.1]{HS},
\[
L_{\cC} = I_{\mathrm{adj}}(\cC) : \left(\prod_{a \in V(\cC)} x_a\right)^{\infty}.
\]
Therefore, there exists a monomial $u \in S_{\cC}$ such that $u f \in I_{\mathrm{adj}}(\cC)$, and hence $u f \in I_{\cC}$. Since $I_{\cC} \subseteq \mathfrak{p}$, we obtain $u f \in \mathfrak{p}$.

Suppose that $f \notin \mathfrak{p}$. Because $\mathfrak{p}$ is prime and $u f \in \mathfrak{p}$, it follows that $u \in \mathfrak{p}$. The primality of $\mathfrak{p}$ then forces at least one variable dividing $u$ to lie in $\mathfrak{p}$.

On the other hand, the inclusion $\mathfrak{p} \subseteq L_{\cC}$ implies that this variable also belongs to $L_{\cC}$. This is impossible, since $L_{\cC}$ contains no variables. The contradiction shows that $f \in \mathfrak{p}$, and consequently $L_{\cC} \subseteq \mathfrak{p}$.

It is known from \cite[Theorem~2.1(b)]{ES} that, for a lattice $\Lambda$, the height of its lattice ideal $L_{\Lambda}$ coincides with the rank of $\Lambda$. In the present situation, the rank of $\Lambda_{\cC}$ equals $|\cC|$. Hence,
\[
\operatorname{ht}(I_{\cC}) \leq \operatorname{ht}(L_{\cC}) = |\cC|,
\]
as desired.
\end{proof}

It is worth underlining that the given upper bound for the height of inner $2$-minor ideals of collections of cells is sharp. It is expected that equality holds, that is, $\operatorname{ht}(I_{\cC}) = |\cC|$ (see \cite[Conjecture 3.7]{DinuNavarra2025}), and this has been proved for several classes of collections of cells (see \cite{CistoNavarraVeer2024, DinuNavarra2025, HerzogQureshiShikama2015}).

We are now in a position to prove the necessary condition: if $I_{\cC}$ is a complete intersection, then $\cC$ must be a chessboard.

\begin{proposition}\label{Prop: CI implies chessboard}
Let $\cC$ be a collection of cells. If $I_{\cC}$ is a complete intersection ideal, then $\cC$ is a chessboard.
\end{proposition}

\begin{proof}

 Suppose that $\cC$ is not a chessboard, then there exist two cells in $\cC$ sharing an edge as in the Figure \ref{fig: Two cells sharing an edge}, up to rotation. 

\begin{figure}[h!]
    \centering
    \includegraphics[scale=0.07]{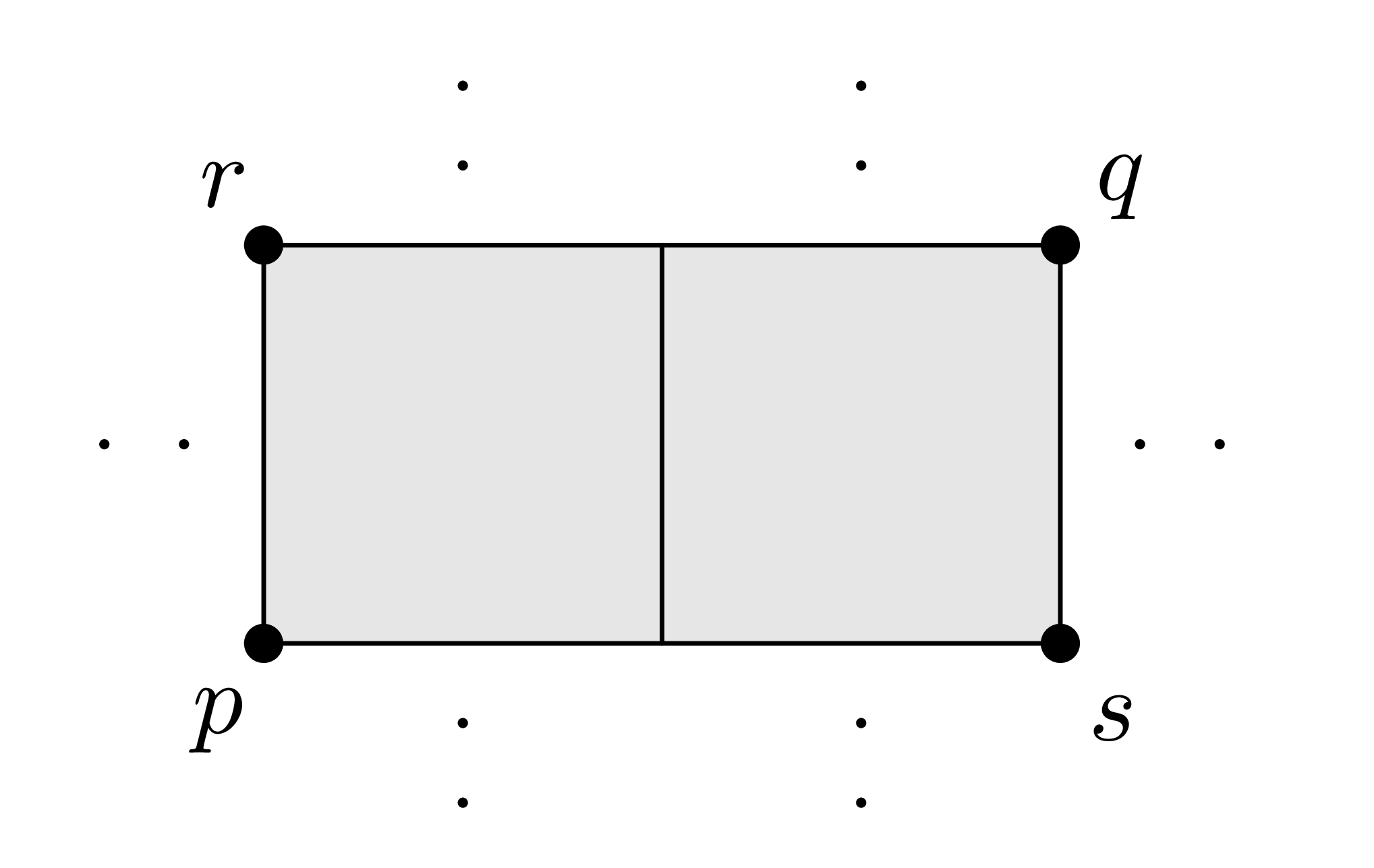}
    \caption{Two cells sharing an edge}
    \label{fig: Two cells sharing an edge}
\end{figure}

Then the number of minimal generators of $I_{\cC}$ is strictly greater than $|\cC|$, because $\cG(I_{\cC})$ contains
\[
\{x_p x_q - x_r x_s\} \cup \{x_a x_b - x_c x_d : [a,b] \text{ is a cell of } \cC \text{ with anti-diagonal corners } c,d\}.
\]
Hence $\mu(I_{\cC}) > |\cC|$. On the other hand, by Proposition~\ref{Prop: LC is a minial prime}, we have $\operatorname{ht}(I_{\cC}) \leq |\cC|$. Hence, it follows that
\[
\operatorname{ht}(I_{\cC}) < \mu(I_{\cC}),
\]
and therefore $I_{\cC}$ is not a complete intersection, which contradicts the assumption that $I_\cC$ is a complete intersection.
\end{proof}

\vskip 0,15 true cm

\noindent\textbf{Acknowledgement}\ \textit{RD is supported by a Return Humboldt fellowship and by a grant of the Ministry of Education and Research, CNCS-UEFISCDI, project number PN-IV-P2-2.1-BSM-2024-0022, within PNCDI IV. FN is supported by Scientific and Technological Research Council of Turkey T\"UB\.{I}TAK under the Grant No: 124F113, and is thankful to T\"UB\.{I}TAK for their support. FN is a member of INDAM-GNSAGA and acknowledges its support.}

\vskip 0,65 true cm

\medskip


\salt

\adresa{$^{(1)}$  Simion Stoilow Institute of Mathematics of the Romanian Academy, Calea Grivitei 21, 010702, Bucharest, Romania \\
E-mail: {\tt rdinu\ap imar.ro} }

\adresa{$^{(2)}$ Sabanci University, Faculty of Engineering and Natural Sciences, Orta Mahalle, Tuzla 34956, Istanbul, Turkey\\
E-mail: {\tt francesco.navarra\ap sabanciuniv.edu} }


\begin{thebibliography}{999}



\bibitem{conca3} {\sc A. Conca, J. Herzog}, Ladder determinantal rings have rational singularities, {\it Adv. Math.} {\bf 132} (1997), 120--147.

\bibitem{BH} {\sc W. Bruns, J. Herzog}, {\it Cohen--Macaulay Rings}, Cambridge University Press, 1993.

\bibitem{BV} {\sc W. Bruns, U. Vetter}, {\it Determinantal Rings}, Lecture Notes in Mathematics {\bf 1327}, Springer, 1988.

\bibitem{CistoNavarra2021} {\sc C. Cisto, F. Navarra}, Primality of closed path polyominoes, {\it J. Algebra Appl.} {\bf 22} (2023), 2350055.

\bibitem{CNU1} {\sc C. Cisto, F. Navarra, R. Utano}, Primality of weakly connected collections of cells and weakly closed path polyominoes, {\it Illinois J. Math.} {\bf 66} (2022), 545--563.



\bibitem{CistoNavarraVeer2024} {\sc C. Cisto, F. Navarra, D. Veer}, Polyocollection ideals and primary decomposition of polyomino ideals, {\it J. Algebra} {\bf 641} (2024), 498--529.



\bibitem{DinuNavarra2023} {\sc R. Dinu, F. Navarra}, On the rook polynomial of grid polyominoes, arXiv:2309.01818, to appear in Discrete Mathematics.

\bibitem{DinuNavarra2025} {\sc R. Dinu, F. Navarra}, Non-simple polyominoes of K\"onig type and their canonical module, {\it J. Algebra} {\bf 673} (2025), 351--384.

\bibitem{ES} {\sc D. Eisenbud, B. Sturmfels}, Binomial ideals, {\it Duke Math. J.} {\bf 84} (1996), 1--45.

\bibitem{EneHerzogQureshiRomeo2021} {\sc V. Ene, J. Herzog, A. A. Qureshi, F. Romeo}, Regularity and the Gorenstein property of L-convex polyominoes, {\it Electron. J. Comb.} {\bf 28} (2021), \#P1.50.



\bibitem{HHMoradi} {\sc J. Herzog, T. Hibi, S. Moradi}, Binomial ideals attached to finite collections of cells, {\it Commun. Algebra} (2024), \textbf{52}(11), 4666–4670.

\bibitem{HHO_Book} {\sc J. Herzog, T. Hibi, H. Ohsugi}, {\it Binomial Ideals}, Graduate Texts in Mathematics {\bf 279}, Springer, 2018.


\bibitem{HerzogQureshiShikama2015} {\sc J. Herzog, A. A. Qureshi, A. Shikama}, Gr\"obner bases of balanced polyominoes, {\it Math. Nachr.} {\bf 288} (2015), 775--783.




\bibitem{HS} {\sc S. Ho\c sten, J. Shapiro}, Primary decomposition of lattice basis ideals, {\it J. Symbolic Comput.} {\bf 29} (2000), 625--639.

\bibitem{JahangirNavarra2024} {\sc R. Jahangir, F. Navarra}, Shellable simplicial complexes and switching rook polynomials of frame polyominoes, {\it J. Pure Appl. Algebra} {\bf 228} (2024), 107576.


\bibitem{KoleyKotalVeer2024} {\sc M. Koley, N. Kotal, D. Veer}, Polyominoes and Knutson ideals, arXiv:2411.16364 (2024).

\bibitem{KV} {\sc M. Kummini, D. Veer}, The $h$-polynomial and the rook polynomial of some polyominoes, {\it Electron. J. Comb.} {\bf 30} (2023), \#P2.6.


\bibitem{MRR2020} {\sc C. Mascia, G. Rinaldo, F. Romeo}, Primality of multiply connected polyominoes, {\it Illinois J. Math.} {\bf 64} (2020), 291--304.

\bibitem{Navarra2025} {\sc F. Navarra}, Shellable flag simplicial complexes of non-simple polyominoes, {\it Boll. Unione Mat. Ital.} (2025), special collection: Second UMI Meeting of PhD Students.



\bibitem{Qureshi2012} {\sc A. A. Qureshi}, Ideals generated by $2$-minors, collections of cells and stack polyominoes, {\it J. Algebra} {\bf 357} (2012), 279--303.

\bibitem{Survey} {\sc F. Navarra, A. A. Qureshi}, {\it Recent Advances in the Theory of Polyomino Ideals}, {\it Galois J. Algebra} {\bf 1} (2025).

\bibitem{QSS} {\sc A. A. Qureshi, T. Shibuta, A. Shikama}, Simple polyominoes are prime, {\it J. Commut. Algebra} {\bf 9} (2017), 413--422.

\bibitem{QureshiRinaldoRomeo2022} {\sc A. A. Qureshi, G. Rinaldo, F. Romeo}, Hilbert series of parallelogram polyominoes, {\it Res. Math. Sci.} {\bf 9} (2022), 28.

\bibitem{RR} {\sc G. Rinaldo, F. Romeo}, Hilbert series of simple thin polyominoes, {\it J. Algebraic Comb.} {\bf 54} (2021), 607--624.

\bibitem{ZGW1} {\sc L. Zheng, J. Guo, T. Wu}, Nonsimple collections of cells whose inner $2$-minor ideals are prime II, {\it J. Algebraic Comb.} {\bf 63} (2026), 15.


\end{thebibliography}
\end{document}